\documentclass[12pt]{article}
\usepackage{amssymb,latexsym}
\usepackage{amsmath}
\usepackage{tikz}
\usepackage{verbatim}
\usetikzlibrary {positioning}
  \usetikzlibrary{arrows,topaths}
       \usetikzlibrary{decorations,backgrounds,snakes}
\usepackage{rawfonts}
\usepackage{amsmath, psfrag}
\usepackage{amsfonts}
\usepackage{amsthm}
\usepackage{lineno,hyperref}

\usepackage{verbatim}
\usepackage {tikz}
\usetikzlibrary {positioning}
\definecolor {processblue}{cmyk}{0.96,0,0,0}

\usepackage{latexsym}
\newtheorem{theorem}{Theorem}
\newtheorem{lemma}[theorem]{Lemma}

\newtheorem{proposition}[theorem]{Proposition}
\def\qed{\ifhmode\unskip\nobreak\fi\quad\ifmmode\Box\else$\Box$\fi}

\begin{document}
\title{Escaping from a quadrant of the $6\times 6$ grid\\
by edge disjoint paths}
\author{{\sl Adam S. Jobson} \\ University of Louisville\\ Louisville, KY 40292
\and {\sl Andr\'e E. K\'ezdy}\\ University of Louisville\\ Louisville, KY 40292
\and {\sl Jen\H{o} Lehel} \\ University of Louisville\\ Louisville, KY 40292
\\
and 
\\
Alfr\'ed R\'enyi Mathematical Institute, \\
Budapest, Hungary}

\maketitle

\begin{abstract}
Let $G$ be the Cartesian product of two finite 
paths, called a grid, and let $T$ be the set of eight distinct vertices of $G$, called terminals. Assume that $T$ is partitioned into four terminal pairs 
$\{s_i,t_i\}$, $1\leq i\leq 4$, to be linked in $G$ by using edge disjoint paths.  To prove that such a linkage always exists 
we need a sequence of  technical lemmas making possible for some terminals to  `escape' from a $3\times 3$ corner of $Q\subset G$, called a `quadrant'. Here we state those lemmas, and give a proof  for the cases when $Q$ contains at most $4$ terminals. 
\end{abstract}

\section{Introduction}

Let $P_k$ be a path with $k$ vertices. The Cartesian product of two $k$-paths,
$P_k\Box P_k$, defines a $k\times k$ grid.  {The vertices of the grid $P_k\Box P_k$ are represented as elements $(i,j)$ of a matrix arranged 
in rows $A(i)$ and columns $B(j)$, $1\leq  i,j\leq k$, where two vertices, $(i,j)$ and $(p,q)$, are adjacent if and only if $|p-i|+|q-j|=1$.

 Here we are dealing with the $6\times 6$ grid $G=P_6\Box P_6$.
Let $T=\{s_1,t_1,s_2,t_2,s_3,t_3,$ $s_4,t_4\}$ be the set of eight distinct vertices of $G$, called {\it terminals}. The set $T$ is partitioned into
four terminal pairs,  $\pi_i=\{s_i,t_i\}$, $1\leq i\leq 4$, to be linked in $G$ by edge disjoint paths. A (weak) {\it linkage} for $\pi_i$,  $1\leq i\leq 4$, means a set of edge disjoint $s_i,t_i$-paths $P_i\subset G$.
  
  In \cite{66} we prove that there exists such linkage, for every choice of $T$, that is the $6\times 6$ grid $G$ is $4$-path pairable. 
  The proof starts with  partitioning the grid  $G$ into four $3\times 3$ grids, called {\it quadrants}. A laborious case analysis in \cite{66} discusses the linkage between the terminals lying in the same or in distinct quadrants. The proof uses a 
  a sequence of  technical lemmas making possible for some terminals to  `escape' from a quadrant $Q\subset G$.
  
We say that a set of  terminals in a quadrant $Q\subset G$ {\it escape} from $Q$
 if there are pairwise edge disjoint `mating paths' from the terminals into distinct mates (exits) located at the union of a horizontal  and a vertical boundary line of $Q$ leading to neighbouring quadrants.
 
 A quadrant $Q\subset G$  is considered to be `crowded', if it contains five or more terminals.
  Among the technical lemmas used in  \cite{66} the proof of three lemmas pertaining to crowded quadrants is presented in \cite{heavy}. It is worth noting that the  lemmas for  crowded quadrants are also applied in \cite{infty}, where it is verified that the infinite grid is $4$-path pairable. Here in Section \ref{crowded}
 we just restate these lemmas without proof.   In Section \ref{light} we state and prove the technical lemmas for `sparse' quadrants containing at most  four terminals. 

\section{Escaping from a crowded quadrant}
\label{crowded}
In the proof of the $4$-path pairability of $P_6\Box P_6$ in \cite{66}
and  that of  the infinite grid $P_\infty\Box P_\infty$ in \cite{infty}, we needed a sequence of technical lemmas to escape terminals  from a $3\times 3$ subgrid $Q$. The proof of three technical lemmas is presented  in \cite{heavy} when the quadrant contains  five or more terminals. Here we just restate them without proof.

Let $G\cong P_6\Box P_6$, let $Q$ be a quadrant of $G$, and let $T\subset G$ be the union of four pairwise disjoint  terminal pairs 
$\{s_i,t_i\}$, $1\leq i\leq 4$.
Let $A$ be a horizontal and  let $B$ be a vertical boundary line
of $Q$.  For a subgraph $S\subseteq G$ set $\|S\|=|T\cap S|$.

 \begin{lemma}
\label{heavy78}
 If $\|Q\|=7$ or $8$,
then there is
 a linkage for two or more pairs in $Q$, and there exist edge disjoint escape  paths for the unlinked terminals into distinct 
 exit vertices in $A\cup B$.\qed
\end{lemma}

 \begin{lemma}
   \label{heavy6}
If $\|Q\|=6$, then there is a linkage for one or more pairs in $Q$, and there exist edge disjoint escape paths for the unlinked terminals into distinct exit vertices of $A\cup B$ such that  $B\setminus A$ contains at most one exit.\qed
    \end{lemma}

  \begin{lemma}
   \label{heavy5}
 If $\|Q\|=5$ and $\{s_1,t_1\}\subset Q$,
then there is
 an $s_1,t_1$-path  $P_1\subset Q$, 
 and the complement of $P_1$ contains edge disjoint escape paths for  the three unlinked terminals into distinct exit vertices of $A\cup B$ such that $B\setminus A$ contains at most one exit.\qed
   \end{lemma}

\section{Escaping from sparse quadrants}
\label{light}
In the proof of the $4$-path pairability of $G\cong P_6\Box P_6$ in \cite{66} 
 we needed a sequence of technical lemmas to escape terminals  from a 
 $3\times 3$ quadrant $Q\subset G$. The lemmas for crowded quadrants are proved in \cite{heavy} and restated in Section \ref{crowded}. Here we state and prove the lemmas
 pertaining to sparse quadrants containing at most four terminals.

Let $G\cong P_6\Box P_6$, let $Q$ be a quadrant of $G$, and let $T\subset G$ be the union of four pairwise disjoint  terminal pairs 
$\{s_i,t_i\}$, $1\leq i\leq 4$.
 W.l.o.g. we may assume that $Q$ is the upper left quadrant of $G$, and thus
   $A=A(3)\cap Q$ and $B=B(3)\cap Q$ are the horizontal and vertical boundary lines, respectively, adjacent to neighbouring quadrants of $G$.
   
    For a vertex set $S\subset V(G)$, $H-S$ is interpreted as the subgraph obtained by the removal of $S$ and the incident edges from $H$;  $x\in H$ simply means a vertex of $H$. Mating (or shifting) a terminal $w$  to vertex $w^\prime$, called a mate of $w$, means specifying a   $w,w^\prime$-path called a {\it mating path}.  
 
 \subsection{Quadrants with two or three terminals}
Finding a linkage for two pairs are facilitated using the property of a graph being `weakly $2$-linked' defined in  \cite{T},
and by introducing the concept of a `frame' in \cite{66}. 

A graph $H$ is weakly $2$-linked, if for every $u_1,v_1,u_2,v_2\in H$, not necessarily distinct vertices, there exist edge disjoint $u_i,v_i$-paths in $H$, for $i=1,2$. A weakly $2$-linked graph must be $2$-connected, but $2$-connectivity is not a sufficient condition.
 The next lemma lists a few weakly $2$-linked subgrids (the simple proofs are omitted). 

\begin{lemma}
\label{w2linked}
The grid $P_3\Box P_k$,
and the subgrid of
$P_k\Box P_k$ induced by 
$(A(1)\cup A(2)) \cup$ $ (B(1)\cup B(2))$ is weakly $2$-linked, for $k\geq 3$.
 \qed
\end{lemma}
  Let $C\subset G$ be a cycle and let $x$ be a fixed vertex of $C$. Take two edge disjoint paths from a member
  of $\pi_j$ to $x$, for $j=1$ and $2$, not using edges of $C$. Then we say that
  the subgraph of the union of $C$ and the two paths to $x$ define a {\it frame} $[C,x]$, for $\pi_1,\pi_2$. A frame
  $[C,x]$, for $\pi_1,\pi_2$,  helps find a linkage
  for the pairs $\pi_1$ and $\pi_2$; in fact, it is enough to mate the other members of the terminal pairs onto $C$  using mating paths edge disjoint from $[C,x]$ and each other.
  
  The concept of a frame was introduced in \cite{66} to facilitate `communication' between quadrants of $G$. For this purpose frames in $G$ can be built on two standard cycles $C_0, C_1\subset G$ as follows.
  
  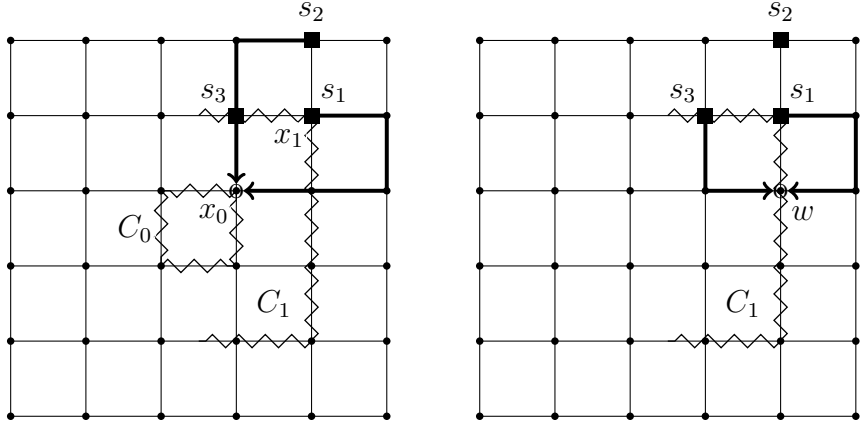
\begin{figure}[htp]
  \begin{center}
      \tikzstyle{T} = [rectangle, minimum width=.1pt, fill, inner sep=3pt]
 \tikzstyle{V} = [circle, minimum width=1pt, fill, inner sep=1pt]
\tikzstyle{B} = [rectangle, draw=black!, minimum width=4pt, inner sep=4pt]
\tikzstyle{txt}  = [circle, minimum width=1pt, draw=white, inner sep=0pt]
\tikzstyle{Wedge} = [draw,line width=1.5pt,-,black!100]

\begin{tikzpicture}
 \foreach \x in {0,1,2,3,4,5} 
    \foreach \y in {0,1,2,3,4,5} 
      { \node[V] () at (\x,\y) {};}
      \foreach \u in {0,1,2,3,4} 
    \foreach \v in {1,2,3,4,5} 
   { \draw (\u,\v) -- (\u+1,\v); \draw (\u,\v-1) -- (\u,\v); } 
  \draw (0,0) -- (5,0) -- (5,5);	
     \draw[line width=.5pt,snake=zigzag]  (2,2) -- (3,2) -- (3,3) -- (2,3) -- (2,2);
      \draw[Wedge,->](4,4)--(5,4) -- (5,3) -- (3.1,3);
     \draw[Wedge,->] (4,5) -- (3,5) -- (3,3.1)  ;
     \draw[line width=.5pt,snake=zigzag]  (2.5,4) -- (4,4) -- (4,1) -- (2.5,1);
     
     \node[T](s1) at (4,4){};
      \node[txt] () at   (4.3,4.3){$s_1$}; 
         \node[T,label=above:$s_2$]() at (4,5){};
           \node[T]() at (3,4){};
           \node[txt] () at   (2.7,4.3){$s_3$}; 
           \node[]() at (3,3){o}; 
       \node[txt] () at   (2.7,2.7){$x_0$}; 
          \node[txt] () at   (3.7,3.7){$x_1$}; 
        \node[txt] () at   (1.65,2.5){$C_0$};  
           \node[txt] () at   (3.5,1.5){$C_1$};        
\end{tikzpicture}
\hskip1truecm
\begin{tikzpicture}
 \foreach \x in {0,1,2,3,4,5} 
    \foreach \y in {0,1,2,3,4,5} 
      { \node[V] () at (\x,\y) {};}
      \foreach \u in {0,1,2,3,4} 
    \foreach \v in {1,2,3,4,5} 
   { \draw (\u,\v) -- (\u+1,\v); \draw (\u,\v-1) -- (\u,\v); } 
  \draw (0,0) -- (5,0) -- (5,5);	
 
      \draw[Wedge,->](4,4) -- (5,4) -- (5,3) -- (4.1,3);
     \draw[Wedge,->]  (3,4) -- (3,3) -- (3.9,3) ;
     \draw[line width=.5pt,snake=zigzag]  (2.5,4) -- (4,4) -- (4,1) -- (2.5,1);
     
        \node[T](s1) at (4,4){};
      \node[txt] () at   (4.3,4.3){$s_1$}; 
      \node[T,label=above:$s_2$]() at (4,5){}; 
         \node[T]() at (3,4){};
           \node[txt] () at   (2.7,4.3){$s_3$}; 
          \node[]() at (4,3){o};  
      \node[txt] () at   (4.3,2.7){$w$};  
       \node[txt] () at   (3.5,1.5){$C_1$};  
\end{tikzpicture}
\end{center}
\caption{Framing $[C_0,x_0]$ for $\pi_1,\pi_2$, and $[C_1,w]$, for $\pi_1,\pi_3$}
\label{frames}
		\end{figure}
 Let $C_0$ be the innermost $4$-cycle of $G$  induced by $(A(3)\cup A(4))\cap (B(3)\cup B(4))$, and let $C_1$ be the $12$-cycle around $C_0$ induced by the neighbors of $C_0$.
 Given a quadrant $Q$ we usually set $x_0=Q\cap C_0$ and we denote by $x_1$ the middle vertex of the path $Q\cap C_1$.  (For instance, in the upper right quadrant of $G$, $x_0=(3,4)$ and $x_1=(2,5)$.)
 
 Let  $\alpha\in\{0,1\}$ be fixed, assume that there are two terminals in a quadrant $Q$ belonging to distinct pairs, say $s_1\in\pi_1$, $s_2\in \pi_2$,  and  let $w\in Q\cap C_\alpha$. We say that $[C_\alpha,w]$ is a {\it framing  in $Q$ for $\pi_1,\pi_2$ to $C_\alpha$}, if 
 there exist edge disjoint mating paths in $Q$
 from $s_1$ and from $s_2$ to $w$,  edge disjoint from $C_1$
 (see examples in Fig.\ref{frames} for framing in the upper right quadrant).
 
    \begin{lemma}
 \label{frame}  
 Let  $s_1\in\pi_1,s_2\in\pi_2$ be two (not necessarily distinct) terminals/mates  in a quadrant $Q$.
 
 (i) For any mapping $\gamma:\{s_1,s_2\}\longrightarrow\{C_0,C_1\}$,
 there exist edge disjoint  mating paths in $Q$ from $s_j$ to vertex $s_j^\prime\in\gamma(s_j)$, $j=1,2$, not using edges of $C_1$.
 
 (ii)  For any fixed $\alpha\in\{0,1\}$, there is a framing 
 $[C_\alpha,x_\alpha]$, 
 for $\pi_1, \pi_2$, where $x_\alpha\in C_\alpha\cap Q$ and the mating paths are in $Q$. 
  \end{lemma}
\begin{proof}
Let $x_2\in Q$ be the corner vertex of the quadrant that has degree $2$ in $G$. If $x_i$ is terminal free for  $i=1$ or $2$, then let $D$ be the Hamiltonian cycle of $Q-x_i$. If $\{x_1,x_2\}=\{s_1,s_2\}$,
then let $D$ be a Hamiltonian path of $Q$ from $x_1$ to $x_2$.
Observe that $D$ does not use edges of $C_1$.
Let $P\subset D$ be the $s_1,s_2$-path of $D$ that contains $x_0$. Thus
$P$ defines a framing for $\pi_1,\pi_2$ with $C_0$.
Since $P$ has a vertex $w\in C_1$, 
 we obtain the required framing  $[C_1,w]$.
 Since the two neighbors of $x_0$ are in $C_1$ claim (i) also follows.
\end{proof}

\begin{lemma}
\label{12toCa} Let $s_p,s_q,s_r$  be distinct terminals in a quadrant $Q$ belonging to three distinct pairs. Then there is a framing in $Q$ for $\pi_p,\pi_q$ to $C_\alpha$, for some $\alpha\in\{0,1\}$, and there is an edge disjoint mating path in $Q$ from $s_r$ to $C_\beta$, 
where $\beta=\alpha+1\pmod 2$, and edge disjoint from $C_1$.
\end{lemma}
\begin{proof} 
W.l.o.g. assume that $p=1,q=2,r=3$. 
First assume that all terminals lie on the boundary cycle $C$ of $Q$. 
Let $P\subset C$ be the $s_1,s_2$-path along $C$ through $x_0$.
If  $s_3\notin P$ then there is always a path from $s_3$ 
to $x_1$ which is edge disjoint from $P$, furthermore,
$P$ defines a framing  for $\pi_1,\pi_2$ to $C_0$. If $s_3\in P$, then let 
$R$  be the $s_3,x_0$-path in $P$, and let $\overline{P}\subset C$ be the $s_1,s_2$-path
edge disjoint from $R$. Then either $\overline{P}$ intersects $C_1$ at some vertex $w$, or there are edge disjoint mating paths from both $s_1,s_2$ to $x_1$, in each case defining a framing for $\pi_1,\pi_2$ to $C_1$.

Assume now that $x_1$ is a terminal. 
Let $x_1^\prime\in C$ be the corner vertex of $Q$
with degree two in $G$. 
If $\{x_1,x_1^\prime\}=\{s_1,s_2\}$, then a shortest $s_1,s_2$-path $R$ and 
an $s_3,x_0$-path disjoint from $R$ yield
a framing $[C_1,x_1]$ for $\pi_1,\pi_2$ and an edge disjoint mating of
$s_3 $ to $C_0$.
If $s_3\in\{x_1,x_1^\prime\}$ then let $C$ be the hamiltonian cycle of $Q-s_3$. Now we define $P\subset C$ to be the $s_1,s_2$-path along $C$ through $x_0$, and the the claim follows. 
 \end{proof} 

\begin{lemma}
\label{Caforpq}
 Let $s_1,s_2,s_3$  be distinct  terminals in a quadrant $Q$ (belonging to distinct pairs); 
let $y_0\in Q$ be a corner vertex of $Q$ with degree three in $G$, and 
let  $z\in \{x_0,y_0\}$ be a fixed corner vertex of $Q$.  Then,

(i)  for some $1\leq p<q\leq 3$, there is a framing in $Q$ for $\pi_p,\pi_q$ to $C_0$,
 and there is an edge disjoint mating path in $Q$ from the third terminal to $C_1$;
 
(ii)  for some $1\leq p<q\leq 3$, there is a framing in $Q$ for $\pi_p,\pi_q$ to $C_1$,
 and there is an edge disjoint mating path in $Q$ from the third terminal to $z$;

\end{lemma}
\begin{proof}
Claim (i) immediately follows by considering a shortest path $P$ through $x_0$ along the boundary cycle $C$ of $Q$, say between terminals $s_p$ and $s_q$. Since the third terminal  $s_r$ is not in $P$ it  is easy to find an edge disjoint mating path from $s_r$ into $C_1$.

In the proof of claim (ii) w.l.o.g. we assume that $Q$ is the upper left quadrant of $G$, and $y_0=(3,1)$. (Recall that in this case $x_0=(3,3)$ and $x_1=(2,2)$.) In order to simplify finding the appropriate indices $p$ and $q$ we are allowing to relabel the terminals during the proof; as a result we conclude with $p=1,q=2$, that is the existence of a required framing for $\pi_1,\pi_2$ to $C_1$
together with the mating of $s_3$ to $z$.
Let $C$ be the boundary cycle of $Q$ oriented counterclockwise.
 
If $x_1$ is a terminal, then set $s_2=x_1$, and take the path $P\subset C$ through $x_0$ between the other two terminals. 
Considering the counterclockwise orientation of $P$,
if its starting vertex  belongs to $B(3)$, then label it with $s_1$, otherwise, label it with $s_3$. Let $R\subset C$ be the smallest path containing $s_3,x_0,y_0$ dedicated to mate $s_3$ into $x_0$ or $y_0$.
In each case there is an $s_1,s_2$-path edge disjoint from $R$ leading to the  framing $[C_1,x_1]$ for $\pi_1,\pi_2$, and yielding the mating 
path in $R$ from $s_3$ to $z$ (see in Fig.\ref{twoframes}).

	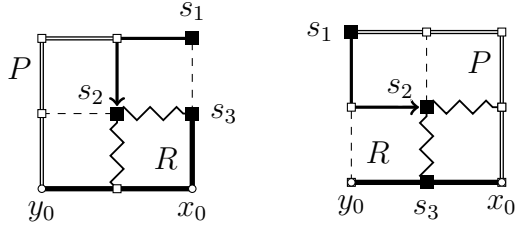
\begin{figure}[htp]
  \tikzstyle{T} = [rectangle, minimum width=.1pt, fill, inner sep=2.5pt]
\tikzstyle{B} = [rectangle, draw=black!, minimum width=1pt, fill=white, inner sep=1.5pt]
\tikzstyle{txt}  = [circle, minimum width=1pt, draw=white, inner sep=0pt]
\tikzstyle{Wedge} = [draw,line width=1.5pt,-,black!100]
\tikzstyle{M} = [circle, draw=black!, minimum width=1pt, fill=white, inner sep=1pt]
\begin{center}
\begin{tikzpicture}

 \draw[dashed] (2,1) -- (2,2) (0,1) -- (1,1);
   \draw[line width=2pt]  (0,0) -- (1,0) -- (2,0) -- (2,1);
   \draw[double,line width=.5pt]       (0,0) --  (0,1) -- (0,2)-- (1,2); 
   \draw[->,line width=1.2pt] (2,2) -- (1,2) -- (1,1.1);
  
    \draw[line width=.7,snake=zigzag]   (1,0) -- (1,1) -- (2,1);  
    
   \node[M]() at (2,0){};     
   \node[B]() at (0,2){};      
   \node[B]() at (1,2){};    
   \node[B]() at (0,1){};   
    \node[B]() at (1,0){};  
    \node[M]() at (0,0){};              
     \node[T]() at (1,1){};                
      \node[T] (x1) at (1,1) {}; 
      
                 \node[T,label=above:$s_1$](s1) at (2,2){};                  
                 \node[T,label=right:$s_3$](s3) at (2,1){};
                 \node[txt] () at (.65,1.25) {$s_2$};    
            
      \node[txt] () at (0,-0.3) {$y_0$};  
      \node[txt] () at (2,-0.3) {$x_0$};  
        \node[txt] () at (1.65,0.4) {$R$}; 
          \node[txt] () at (-.3,1.6) {$P$}; 
          
     	\end{tikzpicture}
	  \hskip.5truecm	
		  \begin{tikzpicture}
     \draw[dashed]  (1,2) -- (1,1) (0,0)--(0,1);
   
    \draw[->,line width=1.2pt] (0,2) -- (0,1)-- (.9,1);
    \draw[line width=.7,snake=zigzag]   (1,0) -- (1,1) -- (2,1);
    
  \draw[double,line width=.5pt]  (1,0)--(2,0) -- (2,2)--(0,2);
   \draw[line width=2pt]  (0,0) -- (1,0) -- (2,0) ; 
  \foreach \u in {0,1}  
   {\node[B]() at (\u,0){};     \node[B]() at (2,\u){};    }        
   \node[B]() at (2,1){};      \node[B]() at (1,2){};    
   \node[B]() at (0,1){};      \node[B]() at (1,0){};  
    \node[M]() at (0,0){};      
    \node[M]() at (2,0){};                           
   \node[T]() at (1,1){};

   \node[B]() at (2,2){};
   
    \node[txt] () at (.65,1.25) {$s_2$};      
    \node[T,label=left:$s_1$](s1) at (0,2){};    
        \node[T,label=below:$s_3$](s3) at (1,0){};  
      \node[txt] () at (0,-0.3) {$y_0$};  
      \node[txt] () at (2,-0.3) {$x_0$};  
       \node[txt] () at (1.7,1.6) {$P$};      
       \node[txt] () at (.35,.4) {$R$}; 
          
     	\end{tikzpicture}	      
	\caption{$x_1$ is a terminal}
	\label{twoframes}
		\end{center}
	\end{figure}

Now we may assume that all terminals are on $C$.
If $x_0$ is a terminal, label it with $s_1$ and let $R$ be the smallest path along $C$ containing $x_0,y_0$ and another terminal; label it with $s_3$. Clearly the
 path from $s_1$ to the remaining terminal $s_2$ not using edges of $R$ intersects $C_1$ at a vertex $w$.  Hence $[C_1,w]$ is a framing for $\pi_1,\pi_2$, and $R$ contains the required mating path from $s_3$ to $z$ (see 
Fig.\ref{onC} (i)).
	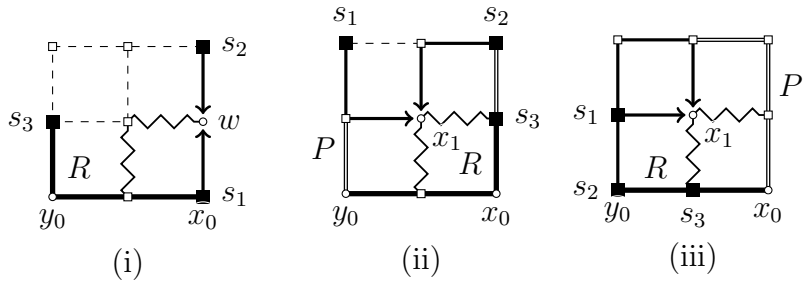
\begin{figure}[htp]
  \tikzstyle{T} = [rectangle, minimum width=.1pt, fill, inner sep=2.5pt]
\tikzstyle{B} = [rectangle, draw=black!, minimum width=1pt,fill=white, inner sep=1.5pt]
\tikzstyle{txt}  = [circle, minimum width=1pt, draw=white, inner sep=0pt]
\tikzstyle{Wedge} = [draw,line width=1.5pt,-,black!100]
\tikzstyle{M} = [circle, draw=black!, minimum width=1pt, fill=white, inner sep=1pt]
\begin{center}

\begin{tikzpicture}
                     
          \draw[dashed]   (2,2) -- (0,2)-- (0,1) -- (1,1) -- (1,2);            
              \draw[->,line width=1.2pt] (2,0) -- (2,.9);
              \draw[->,line width=1.2pt] (2,2) -- (2,1.1);               
    \draw[line width=.7,snake=zigzag]   (1,0) -- (1,1) -- (2,1);
   \draw[line width=2pt] (0,1) -- (0,0) -- (1,0) -- (2,0) ;
      
    \node[B]() at (1,2){};    
    \node[B]() at (0,2){}; 
    \node[B]() at (1,1){};    
        \node[M]() at (0,0){}; 
    \node[B]() at (1,0){};   
       \node[M]() at (2,1){};  
                        
     \node[T,label=right:$s_1$] (s1) at (2,0) {};           
        \node[T,label=left:$s_3$] (s3) at (0,1) {};   
           \node[T,label=right:$s_2$] (s2) at (2,2) {};        
   
  \node[txt] () at (0,-0.3) {$y_0$};  
      \node[txt] () at (2,-0.3) {$x_0$};  
        \node[txt] () at (2.35,1) {$w$};  
            \node[txt] () at (.35,.4) {$R$}; 
            
                 \node[txt]() at (1,-.9){(i)};     
     	\end{tikzpicture}
		 \hskip.5truecm
\begin{tikzpicture}
                
          \draw[dashed]    (1,2) -- (0,2);             
              \draw[->,line width=1.2pt] (0,2) -- (0,1)-- (.9,1);               
              \draw[->,line width=1.2pt] (2,2) -- (1,2)--(1,1.1)  ;   
                                  
    \draw[line width=.7,snake=zigzag]   (1,0) -- (1,1) -- (2,1);
   \draw[double,line width=.5pt] (0,0) -- (0,1) (1,2) (2,2)--(2,1);
  \draw[line width=2pt] (0,0)--(1,0)--(2,0)--(2,1);
    
     \node[B]() at (1,2){};    
     \node[B]() at (0,1){}; 
        \node[M]() at (1,1){};      
        \node[B]() at (1,0){};    
         \node[M]() at (2,0){};      
        \node[M]() at (0,0){};
                        
     \node[T,label=right:$s_3$] (s3) at (2,1) {};           
        \node[T,label=above:$s_1$] (s1) at (0,2) {};   
           \node[T,label=above:$s_2$] (s2) at (2,2) {};        
 
  \node[txt] () at (0,-0.3) {$y_0$};  
      \node[txt] () at (2,-0.3) {$x_0$};  
        \node[txt] () at (1.35,.7) {$x_1$};  
         \node[txt] () at (1.65,0.4) {$R$}; 
             \node[txt] () at (-.3,.6) {$P$}; 
              \node[txt]() at (1,-.9){(ii)};     
     	\end{tikzpicture}
	 \begin{tikzpicture}
             \draw[double, line width=.5pt]  (2,0)-- (2,1)-- (2,2) -- (1,2) ; 
            \draw[->,line width=1.2pt] (0,0)-- (0,2)--(1,2)-- (1,1.1);          
              \draw[->,line width=1.2pt]  (0,1) -- (.9,1);
    \draw[line width=.7,snake=zigzag]   (1,0) -- (1,1) -- (2,1);
   \draw[line width=2pt] (0,0) -- (1,0) -- (2,0) ;
               
   \node[B]() at (0,2){}; 
     \node[B]() at (1,2){};    
       \node[B]() at (2,2){}; 
           \node[B]() at (2,1){}; 
         \node[B]() at (1,0){};    
          \node[M]() at (2,0){};   
           \node[M]() at (1,1){};                                     
     \node[T,label=left:$s_2$] (s2) at (0,0) {};           
        \node[T,label=left:$s_1$] (s1) at (0,1) {};   
           \node[T,label=below:$s_3$] (s3) at (1,0) {};        
              
  \node[txt] () at (0,-0.3) {$y_0$};  
      \node[txt] () at (2,-0.3) {$x_0$};  
          \node[txt] () at (1.35,.7) {$x_1$};  
         \node[txt] () at (.5,0.3) {$R$}; 
          \node[txt] () at (2.3,1.4) {$P$};           
               \node[txt]() at (1,-.9){(iii)};     
     	\end{tikzpicture}
	\caption{Mating $s_3$ into $x_0$ or $y_0$}
	\label{onC}
		\end{center}
	\end{figure}
		
If $x_0$ is not a terminal, 
let $P\subset C$ be the smallest subpath containing $x_0$ and two terminals. We label $s_2$ the terminal missed by $P$. 
For the particular case when all terminals belong to $A(1)\cup B(3)$ let $s_3\in B(3)$;
otherwise, let $s_3\in  (C-(A(1)\cup B(3)))$. Define $R\subset C$ as the smallest path containing $s_3, x_0$ and $y_0$.
 In each case there exist edge disjoint mating paths
from both $s_1,s_2$ to $x_1$ not using edges of $R$.
 Then
  $[C_1,x_1]$ is a framing for $\pi_1,\pi_2$, and  $R$ contains a
   mating path from $s_3$ to $x_0$ or $y_0$ (see Fig.\ref{onC} (ii) and (iii)).
\end{proof}

\begin{lemma}
\label{exit}
Let $A$  be a boundary line of a quadrant $Q\subset G$. Let $Q_0$ be the subgraph obtained by removing the edges of $A$ from $Q$, and
let $Q_i$, $1\leq i\leq 4$, be one of the subgraphs in Fig.\ref{Qmin} obtained from $Q_0$ by edge removal and edge contraction.

(i) For any $H=Q_i$, $1\leq i\leq 4$, and for any three distinct terminals of $H$ there exist edge disjoint mating paths in $H$ from the terminals into not necessarily distinct vertices in $A$.
\begin{figure}[htp]
\tikzstyle{W} = [double,line width=.5,-,black!100]
 \tikzstyle{V} = [circle, minimum width=1pt, fill, inner sep=1pt]
 \tikzstyle{T} = [circle, minimum width=1pt, fill, inner sep=1pt] 
  \tikzstyle{A}  = [rectangle, minimum width=1pt, draw=black,fill=white, inner sep=1.4pt]
\begin{center}
\begin{tikzpicture}
\draw[dashed] (0.5,0.5)--(2.3,0.5)--(2.3,0.9)--(.5,0.9)--(.5,0.5);
\foreach \x in {.7,1.4,2.1}\draw(\x,.7)--(\x,2.1);
\foreach \y in {1.4,2.1}\draw(.7,\y)--(2.1,\y);
\foreach \x in {.7,1.4,2.1}\foreach \y in {.7,1.4,2.1}\node[V] () at (\x,\y){};
\foreach \x in {.7,1.4,2.1}\node[A]()at(\x,.7){};
\node() at (.2,.7){A};
\node() at (1.4,0){$Q_0$};
\end{tikzpicture}
\hskip.5cm
\begin{tikzpicture}
\draw[dashed] (0.5,1.2)--(2.3,1.2)--(2.3,1.6)--(.5,1.6)--(.5,1.2);
\draw (1.4,2.1)--(1.4,1.4)--(2.1,1.4) (.7,1.4)--(1.4,1.4) (.7,.7)--(.7,1.4) (1.4,.7)--(1.4,1.4) (2.1,.7)--(2.1,2.1);
\foreach \x in {.7,1.4,2.1}\foreach \y in {.7,1.4}\node[V] () at (\x,\y){};
\node[T]()at(2.1,2.1){};\node[T]()at(1.4,2.1){};
\foreach \x in {.7,1.4,2.1}\node[A]()at(\x,.7){};
\node() at (2.7,1.4){N};
\node() at (1.4,0){$Q_1$};
\end{tikzpicture}
\begin{tikzpicture}
\draw[dashed] (0.5,1.2)--(2.3,1.2)--(2.3,1.6)--(.5,1.6)--(.5,1.2);
\draw (.7,1.4)--(.7,2.1)--(1.4,1.4) (2.1,1.4)--(.7,1.4) (.7,.7)--(.7,1.4) (1.4,.7)--(1.4,1.4) (2.1,.7)--(2.1,2.1);
\foreach \x in {.7,1.4,2.1}\foreach \y in {.7,1.4}\node[V] () at (\x,\y){};
\foreach \x in {.7,2.1}\node[T] () at (\x,2.1){};
\foreach \x in {.7,1.4,2.1}\node[A]()at(\x,.7){};
\node() at (1.4,0){$Q_2$};
\end{tikzpicture}
\hskip.7cm
\begin{tikzpicture}
\draw[dashed] (0.35,2.15)--
(1.2,1.2)--(2.3,1.2)--(2.3,1.6)-- (1.3,1.6)--
(.8,2.3)--(0.35,2.15);
\draw (.7,2.1)--(2.1,2.1);
\foreach \x in {.7,1.4,2.1}\draw(\x,.7)--(\x,2.1);
\foreach \x in {1.4,2.1}\foreach \y in {.7,1.4}\node[V] () at (\x,\y){};
\foreach \y in {.7,2.1}\node[V] () at (.7,\y){};
\node[T]()at(1.4,2.1){};
\node[T]()at(2.1,2.1){};
\foreach \x in {.7,1.4,2.1}\node[A]()at(\x,.7){};
\node() at (1.4,0){$Q_3$};
\end{tikzpicture}
\begin{tikzpicture}
\draw[dashed] 
(1.6,2.3)--(2.4,1.6)--(2.2,1)--(1.6,1.7)--(1.2,1.7)--(.6,1)
--(.4,1.6)--(1.2,2.3)--(1.6,2.3);
\draw (.7,2.1)--(2.1,2.1);
\foreach \x in {.7,1.4,2.1}\draw(\x,.7)--(\x,2.1);
\foreach \x in {.7,2.1}\foreach \y in {.7,1.4}\node[V] () at (\x,\y){};
\foreach \y in {.7,2.1}\node[V] () at (1.4,\y){};
\node[T]()at(.7,2.1){};
\node[T]()at(2.1,2.1){};
\foreach \x in {.7,1.4,2.1}\node[A]()at(\x,.7){};
\node() at (1.4,0){$Q_4$};
\node() at (0,1.4){N};
\end{tikzpicture}
\caption{Mating into $A$ in adjusted quadrants}
	\label{Qmin}
\end{center}
\end{figure}
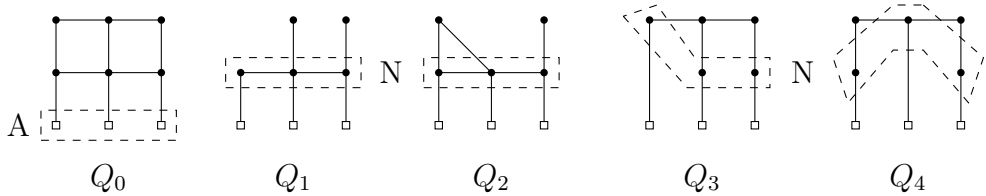

(ii) If  $s_1,t_1,s_2$ are not necessarily distinct terminals in $Q_0$ then there is
 an $s_1,t_1$-path  in $Q_0$ and an edge disjoint mating path from $s_2$ into a vertex of $A$.

(iii) From any three distinct terminals of $Q_0$ there exist pairwise edge disjoint mating paths 
into three distinct vertices of $A$. Furthermore, the claim remains true if two terminals not in $A$ coincide. 
\end{lemma}
\begin{proof} 
W.l.o.g. assume that $Q$ is the upper left quadrant of $G$ and $A=Q\cap A(3)$. Claim (i) can be checked by brief inspection
as follows. If $t$ is a terminal located at a neighbor of some vertex $v\in A$, then $t$ is mated directly into $v$. Let $N$ be the set of neighbors of the vertices in $A$. Assuming that $k$ terminals
belong to $N$, there are $3-k$ terminal-free vertices of $N$, where $1\leq k\leq 3$. Now it is enough to mate the $3-k$ or less terminals not in $N\cup A$ into the terminal-free vertices of $N$. This is trivial for $k=3$, and also for $k=2$, since $H-A$ is connected. For $k=1$, given any two terminal free
vertices $x,y\in N$, there are edge disjoint paths in $H$ from the (at most) two terminals not in $N$ into $x$ and $y$; this can be checked easily for every $H=Q_i$, $1\leq i\leq 4$. 

Claim (iii) is clear provided the terminals are in distinct columns of $Q_0$. If a column contains exactly two terminals not in $A$, then keeping one terminal in that column the other should be mated into a distinct terminal-free column using edges in $A(1)$ or $A(2)$. Similarly, if column $B(j)$ ($1\leq j\leq 3$) contains all terminals, it is enough to mate two coinciding terminals  (if any) not in $A$  into distinct rows $A(1)$ and  $A(2)$, and mating the two distinct terminals/mates into distinct columns along the rows. In each case the mating paths can be extended to $A$ along the distinct columns. 

To see claim (ii),
let $i\in\{1,2\}$ and $q\in\{1,2,3\}$ be such that $s_2\notin A(i)$ and $B(q)\cap \{s_1,t_1\}=\emptyset$. Then we mate the terminals $s_1,t_1$ into $A(i)$ along their columns different from $B(q)$ and connect the mates in $A(i)$ to obtain an $s_1,t_1$-path $P_1$.
 Since $s_2=(j,p)$ with $j=i+1\pmod 2$, it follows that $P_1$ does not use edges of $A(j)$ and $B(q)$. 
Therefore,  we can take an edge disjoint mating path along $A(j)$ from $s_2$  to a vertex in $B(q)$ and extend it along $B(q)$ to $A$. Observe that exactly the same argument works even if $s_2\in \pi_1$. 
\end{proof}
\subsection{Quadrants containing four terminals}
\begin{lemma}
\label{heavy4}
Let $A,B$  be a horizontal and 
a vertical boundary line of quadrant $Q$, let $c$ be the corner vertex of $Q$ not in $A\cup B$, and let $b$ be the middle vertex of $B$ (see  $Q_0$ in Fig.\ref{except}). Denote by $Q_0$  the grid obtained by removing the edges of $A$ from $Q$, and let $T$ be a set of  at most four distinct terminals in $Q_0$. 

(i)  If  $T\subset Q_0-A$ and $c\notin T$, 
then for every terminal $s\in T$, there  is a linkage
in $Q_0$  to connect $s$ to $b$, and  there exist edge disjoint mating paths in $Q_0$
from the remaining terminals of $T$ into not necessarily distinct vertices of $A$. 

\begin{figure}[htp]
\tikzstyle{T} = [rectangle, minimum width=.1pt, fill, inner sep=2.5pt]
 \tikzstyle{V} = [circle, minimum width=1pt, fill, inner sep=1pt]
 \tikzstyle{A}  = [circle, minimum width=.5pt, draw=black,fill=white, inner sep=2.4pt]
  \tikzstyle{B}  = [rectangle, minimum width=1pt, draw=black,fill=white, inner sep=1.4pt]
\begin{center}
\begin{tikzpicture}
\draw[dashed] (0.5,0.5)--(2.3,0.5)--(2.3,.9)--(.5,.9)--(.5,0.5);
\foreach \y in {1.4,2.1}\draw (.7,\y)--(2.1,\y);
\foreach \x in{.7,1.4,2.1}\draw (\x,.7)--(\x,2.1);
\foreach \x in {.7,1.4,2.1}\foreach \y in {1.4,2.1}\node[V] () at (\x,\y){};
\foreach \x in{.7,1.4,2.1}\node[B] () at (\x,.7){};
\node() at (2.1,2.4){B};
\node() at (.3,1){A};
\node() at (1.4,0){$Q_0$};
\node[B,label=left:$c$] () at(.7,2.1){};
\node[T] () at(2.1,1.4){};
\node() at (2.3,1.4){$b$};

\end{tikzpicture}
\hskip1cm
\begin{tikzpicture}
\foreach \y in {1.4,2.1}\draw (.7,\y)--(2.1,\y);
\foreach \x in{.7,1.4,2.1}\draw (\x,.7)--(\x,2.1);
\foreach \x in {.7,1.4,2.1}\foreach \y in {1.4,2.1}\node[V] () at (\x,\y){};
\foreach \x in{.7,1.4,2.1}\node[B] () at (\x,.7){};
\node[T,label=left:$s_{3}$] () at (.7,.7){};
\node[T,label=left:$s_{2}$] () at (.7,1.4){};
\node[T,label=left:$s_1$] () at (.7,2.1){};
 \node() at (1.4,0){$T_1$};
\end{tikzpicture}
\hskip1cm
\begin{tikzpicture}
\foreach \y in {1.4,2.1}\draw (.7,\y)--(2.1,\y);
\foreach \x in{.7,1.4,2.1}\draw (\x,.7)--(\x,2.1);
\foreach \x in {.7,1.4,2.1}\foreach \y in {1.4,2.1}\node[V] () at (\x,\y){};
\foreach \x in{.7,1.4,2.1}\node[B] () at (\x,.7){};
\node[T,label=above:$s_{2}$] () at (2.1,2.1){};
\node[T,label=above:$s_{3}$] () at (1.4,2.1){};
\node[T,label=above:$s_{4}$] () at (.7,2.1){};
\node[T] () at(.7,2.1){};
\node[T] () at(2.1,1.4){};
\node() at (1.9,1.2){$s_1$}; 
 \node() at (1.4,0){$T_2$};
\end{tikzpicture}
\caption{Projection to $A$}
	\label{except}
\end{center}
\end{figure}
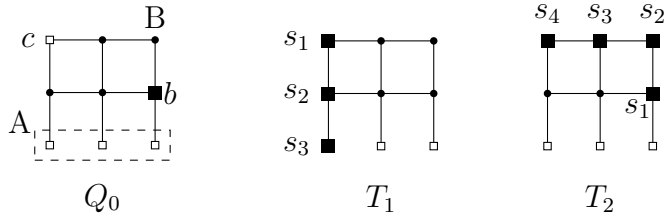
(ii) If $T$ is different from $T_1$ and 
$T_2$  in Fig.\ref{except}, then
for $\min\{3,|T|\}$ choices of a terminal $s\in T$, 
there is a linkage
in $Q_0$  to connect $s$ to $b$, and  there exist edge disjoint mating paths in $Q_0$
from the remaining terminals of $T$ into not necessarily distinct vertices of $A$.

(iii)
If $T$ is one of $T_1$ and 
$T_2$  in Fig.\ref{except}, then the claim  in (ii) above is true only for $s=s_1$ and $s_2$.
\end{lemma}
\begin{proof} W.l.o.g. we assume that $Q$ is the upper left quadrant of $G$, thus $A=A(3)\cap Q$, $B=B(3)\cap Q$, $c=(1,1)$ and $b=(2,3)$. Let $T=\{s_1,s_2,s_3,s_4\}$. We say that $s_i$ satisfies the claim provided there is an 
 $s_i,b$-path in $Q_0$, and  there exist edge disjoint mating paths in $Q_0$
from the remaining terminals of $T$ into $A$. For the set of terminals $T_1$ and $T_2$ in Fig.\ref{except}, terminals 
$s_3,s_4$ do not satisfy the claim, meanwhile $s_1, s_2$ do satisfy it, thus (iii) is obvious.\\

(i) Assume that $T\subset A(1)\cup A(2)$ and $c\notin T$. 
If $s\in A(2)\cup B$, then we take the $s,b$-path in $A(2)\cup B$, and use distinct columns to mate the remaining terminals to (distinct) vertices of $A$, thus $s$ satisfies the claim. 
Assume now that $s=(1,2)\in T$ and take the 
$s,b$-path $P\subset B(2)\cup A(2)$. In the complement of $P$
there are edge disjoint mating paths to (distinct) vertices of $A$,
for the remaining at most three terminals 
(see Fig.\ref{trQ}(i)). \\

(ii) The claim is trivial for $|T|=1$ and $2$. For $|T|=3$ or $4$ we need to show that (at least) three terminals in $T$ satisfy the claim. By (i), we have two cases to consider: either  $T\cap A\neq\emptyset$ or  $T\cap A=\emptyset$ and $c\in T$. 

Assume first that $\|A\|\geq 1$, and let $s_1=(3,j)$, for some $1\leq j\leq 3$.
Any terminal $s=(i,j)\in A(1)\cup A(2)$ satisfies the claim.
This is obvious for terminals $s\in A(2)\cup (B\setminus A)$, since after taking the shortest 
$s,b$-path $P\subset B(j)\cup A(i)$ at most two terminals not in $A$ remain, and they can be mated into distinct vertices of $A$ along distinct columns. Thus we are done if $\|A(1)\cup A(2)\|=3$. The exceptions where this is not true are: either 
$||A||=1$ and $|T|=3$ or $||A||\geq 2$ and $|T|=4$. 

\begin{figure}[htp]
\tikzstyle{W} = [double,line width=.5,-,black!100]
\tikzstyle{T} = [rectangle, minimum width=.1pt, fill, inner sep=2.5pt]
\tikzstyle{V} = [circle, minimum width=1pt, fill, inner sep=1pt]
   \tikzstyle{B}  = [rectangle, minimum width=1pt, draw=black,fill=white, inner sep=1.4pt]
\begin{center}
\begin{tikzpicture}
\draw[line width=1.5pt](1.4,2.1)--(1.4,1.4)--(2.1,1.4);
\draw[dashed] (0.5,0.5)--(2.3,0.5)--(2.3,0.9)--(.5,0.9)--(.5,0.5);
\draw (.7,.7)--(.7,2.1)--(2.1,2.1) (.7,1.4)--(1.4,1.4)  (1.4,.7)--(1.4,1.4) (2.1,2.1)--(2.1,.7);
\foreach \x in {.7,1.4,2.1}\foreach \y in {1.4}\node[V] () at (\x,\y){};
\foreach \x in {.7,1.4,2.1}\foreach \y in {.7}\node[B] () at (\x,\y){};
\node[B]()at(.7,2.1){};\node[V]()at(2.1,2.1){};
\node[T,label=above:$s$]()at(1.4,2.1){};
\node[label=right:$b$]()at(1.9,1.4){};
\node() at (.2,.7){A};\node() at (1.65,1.65){$P$};
\node() at (1.4,0){(i)};
\end{tikzpicture}
\hskip1cm
\begin{tikzpicture}
\draw[line width=1.5pt](1.4,.7)--(1.4,1.4)--(2.1,1.4);
\draw (.7,.7)--(.7,2.1)--(1.4,2.1) (.7,1.4)--(1.4,1.4)
 (1.4,1.4)--(1.4,2.1)--(2.1,2.1)--(2.1,.7);
\foreach \x in {.7,1.4,2.1}\foreach \y in {1.4}\node[V] () at (\x,\y){};
\foreach \x in {.7,1.4,2.1}\foreach \y in {.7}\node[B] () at (\x,\y){};
\node[V]()at(2.1,2.1){};\node[V]()at(1.4,2.1){};
\node[V]()at(.7,2.1){};\node[T]()at(1.4,.7){};\node()at(1.15,.7){$s$};
\node[label=right:$b$]()at(1.9,1.4){};\node[T]()at(.7,.7){};
\node() at (1.65,1.15){$P$};
\node() at (1.4,0){(ii)};
\end{tikzpicture}
\hskip1cm
\begin{tikzpicture}
\draw[line width=1.5pt](2.1,.7)--(2.1,1.4);
\draw (.7,.7)--(.7,2.1)--(2.1,2.1) (.7,1.4)--(1.4,1.4) -- (1.4,.7) (1.4,2.1) -- (1.4,1.4) -- (2.1,1.4)--(2.1,2.1);
\foreach \x in {.7,1.4,2.1}\foreach \y in {1.4}\node[V] () at (\x,\y){};
\foreach \x in {.7,1.4,2.1}\foreach \y in {.7}\node[B] () at (\x,\y){};
\node[V]()at(2.1,2.1){};\node[V]()at(.7,2.1){};\node[V]()at(1.4,2.1){};
\node[T]()at(2.1,.7){};\node[label=right:$b$]()at(1.9,1.4){};
\node[label=right:$s$]()at(1.95,.7){};\node[T]()at(.7,.7){};
\node() at (1.8,1){$P$};
\node() at (1.4,0){(iii)};
\end{tikzpicture}

\caption{Mating into $A$ in the complement of $P$}
	\label{trQ}
\end{center}
\end{figure}
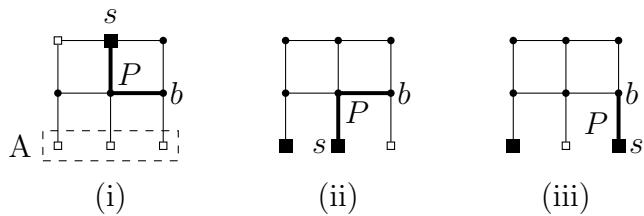

Let $\|A\|\geq 2$. Then we have that either $s=(3,3)$ or $(3,2)$ is a terminal in $T$.  After taking the shortest 
$s,b$-path $P\in B(j)\cup A(2)$, $j=2,3$, 
at most two terminals not in $A$ remain which can be mated in 
the complement of $P$ into distinct vertices of $A$ (see Fig.\ref{trQ}(ii) and (iii)).
Thus we may assume that $||A||=1$ and $|T|=3$. The analysis above shows that for $c\notin T$ or for $(3,1)\notin T$, every $s\in T$ satisfies the claim. If none happens, then the only exception
is $T=A(1)=T_1$.

 The cases not covered so far are $\|A\|=0$, $c\in T$ and $|T|=4$. Terminals 
 in $A(2)\cup B$ clearly satisfy the claim.  Thus we may assume that both  $s=(1,1)$ and $(1,2)$ are terminals in $T$. If  none of them satisfy  the claim, then we obtain easily that $(1,3)\in T$, then $(2,3)\in T$, thus $T=T_2$ follows.
\end{proof}

\begin{lemma}
\label{boundary}
Let $A,B$  be a horizontal and 
a vertical boundary line of a quadrant $Q$.
For every $s_1,t_1,s_2,s_3\in Q$  and $\psi: \{s_2,s_3\} \longrightarrow \{A,B\}$,
there is a linkage for $\pi_1$, and  there exist edge disjoint mating paths in $Q$
from $s_j$, $j=2,3$, to distinct vertices  $s_j^*\in \psi(s_j)$.
\end{lemma}
\begin{proof} 
Terminals  $s_j\in Q$, $j=2,3$, will be called {\it singletons}. A connected subgraph $Y\subseteq Q$ containing a vertex $a\in A$ and $b\in B$ will be called a {\it clamp}
for its vertices with anchor set $\{a,b\}$. Observe that if a singleton $s_j$ is a vertex of $Y$, then there is a mating path from $s_j$
to $\psi(s_j)$. By definition, a path from $s_j$ to the `corner' vertex $x_0\in A\cap B$ is a clamp for $s_j$ with anchor set  $\{x_0\}$. 
In several cases the proof of the lemma consists of decomposing the edge set of $Q$ into
an $s_1,t_1$-path $P_1$ and two disjoint edge disjoint clamps, $Y_2,Y_3$, with disjoint anchor sets.

W.l.o.g. we assume that $Q$ is the upper left quadrant of $G$,  thus $A=A(3)\cap Q$, $B=B(3)\cap Q$,  $x_0=(3,3)$, and $x_1=(2,2)$.
Set $S=Q-(A\cup B)$ and let us call $C=Q-x_1$  the {\it boundary cycle} of $Q$. 
Observe that $Z=(A(2)\cup B(2))\cap Q$, the complement of $C$ in $Q$, is a clamp for its vertices. 

Let  $M=(A(1)\cup B(1))\cup\{x_1\}$. We claim that $s_2,s_3\in M$ can be assumed, otherwise the lemma follows. Suppose, for instance, that we have $s_3\in \{x_0,(3,2),(2,3)\}$. 
Then define $s_3^*=x_0$, and take the mating path $P_3^*= (s_3, x_0)$. Select  any vertex  $s_2^*\in \psi(s_2)\setminus\{x_0\}$,
and shift $x_0$ to its neighbor $y_0\neq s_3$, in case of $x_0\in\pi_1$. 
By Lemma \ref{w2linked}, $Q-x_0$ is  weakly $2$-linked, hence there is an $s_2,s_2^*$-path and an edge disjoint  linkage for 
$\pi_1$ (or an edge disjoint path from $\pi_1\setminus\{x_0\}$ to $y_0$,  if $x_0\in\pi_1$, which becomes a linkage for $\pi_1$ by appending the edge $y_0x_0$). 

	\begin{figure}[htp]
\tikzstyle{T} = [rectangle, minimum width=.1pt, fill, inner sep=2.5pt]
\tikzstyle{B} = [rectangle, draw=black!, minimum width=1pt, fill=white, inner sep=1.5pt]
\tikzstyle{txt}  = [circle, minimum width=1pt, draw=white, inner sep=0pt]
\tikzstyle{Wedge} = [draw,line width=.7pt,-,black!100]
\tikzstyle{M} = [circle, draw=black!, minimum width=1pt, fill=white, inner sep=1pt]

\begin{center}
\begin{tikzpicture}
      \draw[dashed] (0,1) -- (0,2) ;     
\draw[dashed] (2,2) -- (2,1)  ;
\draw[->,line width=1.2pt] (0,2) -- (2,2) -- (2,0.1) ;          
\draw[double,line width=.3,snake=zigzag]   (2,1) -- (1,1); 
\draw[double,line width=.3,snake=zigzag] (1,0) -- (1,2) ;    
\draw[line width=2pt] (1,1)--(0,1)--(0,0)-- (2,0);
  
   \node[B]() at (2,2){};         
   \node[B,label=left:$w$]() at (0,1){};   
   \node[B]() at (1,0){};    
   \node[B]() at (0,0){};    
   \node[T](t1) at (1,1){};    
      \node[B]() at (2,1){};   
            
           \node[T,label=right:$s_3^*$] (s1) at (2,0) {};    
              \node[T,label=above:$s_2$] (s2) at (1,2) {};              
        \node[T,label=left:$s_3$,label=above:$y_1$] (s3) at (0,2) {};

      \node[txt] () at (.75,1.3) {$x_1$};
  \node () at (2,-.3) {$s_1$}; 
   \node[txt] () at (1.25,.75) {$t_1$};   
        \node[txt] () at (1.6,1.65) {$P_3^*$};  
        \node[txt] () at (.35,.35) {$P_1$};  
    
       \node[txt]() at (1,-.9){(i)};        
     	\end{tikzpicture}
	\hskip.5cm
\begin{tikzpicture}

  \draw[->,line width=1.2pt] (0,0) -- (1.9,0);          
   \draw[double,line width=.3,snake=zigzag]  (0,0) -- (0,2) -- (2,2); 
   \draw[double,line width=.5] (0,1) -- (1,1) -- (1,2) (2,2) -- (2,1) ;
        \draw[line width=2pt] (1,0)--(1,1)--(2,1)--(2,0);
        
   \node[B]() at (2,2){};      
   \node[B]() at (1,2){};   
      \node[B]() at (2,1){};      
   \node[T]() at (1,0){}; 
     \node[B] () at (1,1) {}; 
       
               \node[T,label=above:$s_2$](s2) at (0,2){};    
                \node[T,label=right:$s_3^*$] (s1) at (2,0) {};                                  
              \node[T,label=left:$s_3$] (s3) at (0,0) {};
                \node[B] () at (0,1) {};
              
                \node[txt] () at (2,-.3) {$s_1$};
                 \node[txt] () at (1,-.3) {$t_1$};    
        \node[txt] () at (.35,1.65) {$Y$};  
        \node[txt] () at (.55,.35) {$P^*_3$};  
          \node[txt] () at (1.5,.65) {$P_1$};  
                 
       \node[txt]() at (1,-.9){(ii)};        
     	\end{tikzpicture}
	\hskip.5cm
	\begin{tikzpicture}
       \draw[dashed] (0,0) -- (0,1) (2,2)--(2,1);
      \draw[->,line width=1.2pt]  (0,0) -- (1.9,0);          
   \draw[double,line width=.3,snake=zigzag]   (2,2) -- (1,2) -- (1,0) ; 
     \draw[line width=2pt] (1,2)--(0,2) -- (0,1) -- (2,1) -- (2,0);
   
   \node[B]() at (2,2){};      
   \node[B]() at (1,2){};   
       \node[B]() at (1,1){};    
          \node[B]() at (0,1){};  
             \node[B]() at (1,0){};   
                \node[B]() at (0,2){}; 
           \node[T,label=right:$s_3^*$] (s1) at (2,0) {};    
              \node[T,label=right:$s_2$] (s2) at (2,1) {};        
              \node[T,label=left:$s_3$] (s3) at (0,0) {};
      \node[T,label=above:$t_1$](t1) at (1,2){};   
  \node () at (2,-.3) {$s_1$};

       \node[txt] () at (1.35,1.65) {$Y$};  
     \node[txt] () at (.35,1.35) {$P_1$};  
        \node[txt] () at (.55,.35) {$P^*_3$};  
 
       \node[txt]() at (1,-.9){(iii)};        
     	\end{tikzpicture}
	\caption{Clamps}
	\label{clamp}
		\end{center}
	\end{figure}
	
Case 1: $x_0\in \pi_1$. 	
Let $s_1=x_0$, and assume that one of the singletons belong to $Z$, say $s_2\in Z$. If  $t_1,s_3\in C$, then let $W\subset C$ be the  $t_1,s_3$-path containing   $x_0$. If $x_1\in\{t_1,s_3\}$, then let $y_1\in\{t_1,s_3\}\setminus\{x_1\}$, and $w\in A(1)\cup B(1)$ be the neighbor of $x_1$ that is different from $s_2$. Now we define $W$ as a walk starting 
with the edge $x_1w$, then going through $x_0$ around $C$ and ending up at $y_1\in C$.
We define $s_3^*=x_0$, $P^*_3\subset W$ to be the $s_3,s_3^*$-subpath of $W$, and $P_1\subset W$ as the $s_1,t_1$-subpath of $W$. Subgraph $Z$ or ($Z-x_1w$) is a clamp for  $s_2$ (see Fig.\ref{clamp} (i)).

Now we still have $s_1=x_0$, furthermore, none of the singletons is a vertex of $Z$. Since $s_2,s_3\in\{(1,1),(1,3),(3,1)\}$, by symmetry, 
we may assume that $s_3=(3,1)$, and either $s_2=(1,1)$ or  $s_2=(1,3)$.
In the first case the path $P_3^*\subset A$ serves as a mating of $s_3$ into $s_3^*=x_0$, the subgraph $A(1)\cup B(1)$ is a clamp $Y$ for $s_2$,  and the linkage for $\pi_1$ is obtained in the complement of $Y\cup P_3^*$ (see Fig.\ref{clamp} (ii)). If  $s_2=(1,3)$, then by symmetry, we may also assume that
$t_1\notin A$. Then the mating path $P_3^*$ defined as before,
the path $Y$ induced by $B(2)\cup\{(1,3)\}$ serves as a clamp for $s_2$, and the path $P_1$ on the vertices $A(2)\cup\{x_0,(1,2)\}$ is a linkage for $\pi_1$
 (see Fig.\ref{clamp} (iii)).
 
  From now on we assume that  $x_0\notin \pi_1$, and $s_2,s_3\in M$. 
In the discussion which follows we use the following obvious property which makes possible to find a matching between two edge disjoint clamps $Y_2, Y_3\subset Q$ and the two singletons $s_2,s_3$.
\begin{proposition}
\label{partition}
 For $s_1,t_1,s_2,s_3\in Q-x_0$, let $P_1$ be an $s_1,t_1$-path, and 
let $Y_2,Y_3\subset Q$ be edge disjoint clamps in the complement of $P_1$
with disjoint anchor sets.
If  both $Y_2-Y_3$ and $Y_3-Y_2$ contains at most one terminal among 
$\pi_0=\{s_2,s_3\}$,
then there is a matching  $\gamma:\pi_0\longrightarrow\{Y_2,Y_3\}$
such that $s_j\in\gamma(s_j)$, for $j=1,2$.\qed
\end{proposition}
Case 2: $S$ and $A\cup B$ is not spanned by $ \pi_1$, that is either $\pi_1\subset S$  or $\pi_1\subset A\cup B$. 
 
Let $\pi_1\subset S$. (We notice that due to the symmetry of rows and columns, the forthcoming arguments remain valid when swapping $A$ and $B$.)
 
If $\pi_1\subset A(1)$, then let $P_1=s_1t_1$. We define 
the clamp $Y_2$ as the $5$-path $(3,1)- (2,1)-(2,2) - (1,2) - (1,3)$ with the end vertices as its anchor set; let clamp $Y_3$ be defined as the complement of $Y_3$ in $Q-(1,1)$ anchored at  $\{x_0\}$ (see Fig.\ref{pi1inS} (i)). 
Since $\pi_0\cap (Y_2- Y_3)=\{(2,1)\}$, and 
$\pi_0\cap (Y_3- Y_2)=\emptyset$, furthermore, their anchor sets  are disjoint, Proposition \ref{partition} applies. 

 	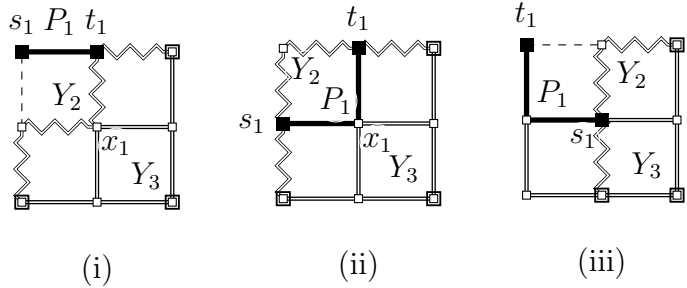
\begin{figure}[htp]
 \tikzstyle{T} = [rectangle, minimum width=.1pt, fill, inner sep=2.5pt]
\tikzstyle{B} = [rectangle, draw=black!, minimum width=1pt, fill=white, inner sep=1.5pt]
\tikzstyle{txt}  = [circle, minimum width=1pt, draw=white, inner sep=0pt]
\tikzstyle{Wedge} = [draw,line width=.7pt,-,black!100]
\tikzstyle{M} = [circle, draw=black!, minimum width=1pt, fill=white, inner sep=1pt]
\tikzstyle{A}  = [rectangle, line width=.7pt, draw=black, inner sep=2.5pt]

\begin{center}
\begin{tikzpicture}

   \draw[double,line width=.5pt]  (0,0) -- (2,0)-- (2,2) (1,0)--(1,1) -- (2,1) ; 
   \draw[double, line width=.3pt,snake=zigzag]  (0,0) -- (0,1) -- (1,1) -- (1,2) -- (2,2); 
   \draw[line width=2pt] (0,2) -- (1,2);
       \draw[dashed] (0,1)--(0,2);  
   \node[B]() at (2,2){};      
   \node[B]() at (0,1){};     
  \node[A]() at (2,0){}; 
   \node[A]() at (0,0){}; \node[A]() at (2,2){};   
         \node[B]() at (1,1){};      
              \node[B]() at (0,0){}; 
                    \node[B]() at (2,1){}; 
                    
             \node[T,label=above:$t_1$](t1) at (1,2){};              
                   \node[B](x0) at (2,0){};               
      \node[B] () at (1,0) {};                     
         \node[T,label=above:$s_1$](s1) at (0,2){};  
      \node[txt] () at (1.255,.75) {$x_1$};                
        
  
         \node[txt]() at (.5,2.4){$P_1$};
               \node[txt]() at (1.65,.35){$Y_3$};
               \node[txt]() at (.6,1.4){$Y_2$};
         \node[txt]() at (1,-.9){(i)};            
     	\end{tikzpicture}
	\hskip.5cm
	\begin{tikzpicture}

   \draw[double,line width=.5pt]  (0,0)--(2,0)-- (2,2) (1,0) -- (1,1) -- (2,1); 
   \draw[double, line width=.3pt,snake=zigzag]   (0,0)--(0,1) -- (0,2) -- (1,2) -- (2,2); 
   \draw[line width=2pt] (0,1)--(1,1) -- (1,2);
         
   \node[B]() at (2,2){};      
   \node[B]() at (0,2){};    
         \node[B]() at (1,1){};      
              \node[B]() at (1,0){}; 
                    \node[B]() at (2,1){}; 
                    
             \node[T,label=left:$s_1$](s1) at (0,1){};              
                   \node[B]() at (2,0){};               
      \node[B] () at (0,0) {};                     
         \node[T,label=above:$t_1$](t1) at (1,2){};  
         \node[A]() at (2,0){}; 
   \node[A]() at (0,0){}; \node[A]() at (2,2){}; 
            \node[txt] () at (1.255,.75) {$x_1$};                    
         \node[txt]() at (.3,1.7){$Y_2$};
               \node[txt]() at (.7,1.3){$P_1$};
                 \node[txt]() at (1.6,.4){$Y_3$};
         \node[txt]() at (1,-.9){(ii)};            
     	\end{tikzpicture}
	\hskip.7truecm
 	\begin{tikzpicture}

   \draw[double,line width=.5pt] (0,1) -- (0,0) -- (2,0) -- (2,2) (1,1)--(2,1); 
   \draw[double, line width=.3pt,snake=zigzag]   (1,0)-- (1,2) -- (2,2); 
   \draw[line width=2pt] (0,2)--(0,1) -- (1,1);    
     \draw[dashed] (0,2)--(1,2) ;
         
   \node[B]() at (0,1){};       
         \node[T](s1) at (1,1){};      
          \node[B]() at (1,0){}; 
           \node[B]() at (2,1){}; 
          \node[B] (t1) at (0,0) {};                                
           \node[B]() at (2,2){};              
            \node[B]() at (2,0){};                              
         \node[B]() at (1,2){}; 
       \node[A]() at (2,0){}; 
   \node[A]() at (1,0){}; \node[A]() at (2,2){};        
      \node[txt] () at (.75,.75) {$s_1$};                
           \node[T,label=above:$t_1$]() at (0,2){}; 
              
         \node[txt]() at (1.4,1.6){$Y_2$};
               \node[txt]() at (.35,1.35){$P_1$};
               \node[txt]() at (1.6,.4){$Y_3$};
         \node[txt]() at (1,-.9){(iii)};            
     	\end{tikzpicture}
 
	\caption{$\pi_1\subset S$}
	\label{pi1inS}
		\end{center}
	\end{figure}

 If $\pi_1\subset B(2)$, then let $P_1=s_1t_1$. We define 
the clamp $Y_2$ as the $5$-path $(3,1)- (2,1)-(2,2) - (1,2) - (1,3)$ with the end vertices as its anchor set; and let $Y_3$ be the complement of $Y_3$ in $Q$ anchored at $\{x_0\}$. 
Observe that $\pi_0\cap (Y_2- Y_3)=\{(1,1)\}$, and 
$\pi_0\cap (Y_3- Y_2)=\emptyset$, furthermore the anchor sets are disjoint. The same
clamp $Y_2$ works for $\pi_1=\{(1,2),(2,1)\}$, when we set $P_1=
(s_1,x_1,t_1)$, and define the clamps
$Y_2$ and $Y_3$ as before (see Fig.\ref{pi1inS} (ii)). The only change is that 
$\pi_0\cap (Y_3- Y_2)=\{x_1\}$, thus Proposition \ref{partition} applies.

Let $\pi_1=\{x_1,(1,1)\}$. Since $\pi_0\subset (A(1)\cup B(1)) \setminus\{(1,1)\}$,  by symmetry, we may assume that $|\pi_0\cap B(1)|\leq 1$. Define the path $P_1=s_1-(2,1)-t_1$, 
the clamp $Y_2=(1,3)-(1,2)-x_1-(3,2)$ with its end vertices as its anchor set, and let $Y_3=C-\{(1,1),(1,2)\}$ be a clamp anchored at $\{x_0\}$ (see Fig.\ref{pi1inS} (iii)).
Observe that 
$\pi_0\cap (Y_2- Y_3)=\{(1,2)\}$ and 
$\pi_0\cap (Y_3- Y_2)$ contains at most one terminal by our assumption. Thus Proposition \ref{partition} applies.

Let $\pi_1\subset A\cup B$. We define $P_1$ to be the $s_1,t_1$-path in $A\cup B$, let
$Y_2=M\setminus\{x_1\}$, and let $Y_3=Z$. 
Since $\pi_0\subset M$, the only vertices in the difference sets which can hold terminals of $\pi_0$ are
$(1,1)\in (Y_2- Y_3)$ and $x_1\in(Y_3- Y_2)$, thus Proposition \ref{partition} applies.

Case 3: $ \pi_1$ spans between $S$ and $A\cup B$, let $s_1\in S$, $t_1\in A\cup B$. (Due to the symmetry of rows and columns, the forthcoming arguments  remain valid when swapping $A$ and $B$.)

Case 3.1: $\pi_1\subset B(i)$ (or $\pi_1\subset A(i)$), $i=1,2$. For $s_1=(1,1)$ and $t_1=(3,1)$,
let $P_1= B(1)$ be the linkage for $\pi_1$, let  $Y_2=(1,3)-(1,2)-x_1-(3,2)$ be a clamp with the two end vertices as its anchor set, and define clamp $Y_3$ to be the subgraph induced by 
$(A(2)\cap Q)\cup B \cup\{(3,2)\}$ with anchor set $\{x_0\}$  (see Fig.\ref{pi1separated} (i)).

For $s_1=(2,1)$ and $t_1=(3,1)$, let $P_1=s_1t_1$, 
let  $Y_2=(A(1)\cap Q)\cup B$ anchored at $x_0$, and define $Y_3$ as the complement of $P_1\cup Y_2$ (the edges of $A$ might be removed) with anchor set $\{(3,2),(2,3)\}$, see in Fig.\ref{pi1separated} (ii).

For $s_1=(1,2)$ and $t_1=(3,2)$,
let $P_1=B(2)$, let $Y_2=(A(1)\cup B(1))\cap Q$ be a clamp with the two end vertices as its anchor set, and define the clamp
$Y_3$ as the complement of $P_1\cup Y_2$ anchored at $\{x_0\}$ (see Fig.\ref{pi1separated} (iii).
	
 	\begin{figure}[htp]
 \tikzstyle{T} = [rectangle, minimum width=.1pt, fill, inner sep=2.5pt]
\tikzstyle{B} = [rectangle, draw=black!, minimum width=1pt, fill=white, inner sep=1.5pt]
\tikzstyle{txt}  = [circle, minimum width=1pt, draw=white, inner sep=0pt]
\tikzstyle{Wedge} = [draw,line width=.7pt,-,black!100]
\tikzstyle{M} = [circle, draw=black!, minimum width=1pt, fill=white, inner sep=1pt]
\tikzstyle{A}  = [rectangle, line width=.7pt, draw=black, inner sep=2.5pt]

\begin{center}
 	\begin{tikzpicture}

   \draw[double,line width=.5pt] (1,0) -- (2,0) -- (2,2) (0,1)--(2,1); 
   \draw[double, line width=.3pt,snake=zigzag]   (1,0)-- (1,2) -- (2,2); 
   \draw[line width=2pt] (0,2)--(0,0) ;    
     \draw[dashed] (0,2)--(1,2) (0,0)--(1,0);
         
          \node[B]() at (0,1){};       
           \node[A]() at (2,0){};          \node[A]() at (2,2){};   
                 \node[A]() at (1,0){};     
          \node[B]() at (1,0){}; 
           \node[B]() at (2,1){}; 
          \node[B] () at (1,1) {};                                
           \node[B]() at (2,2){};              
            \node[B]() at (2,0){};                          
         \node[B]() at (1,2){};   
         
             \node[T,label=below:$t_1$](t1) at (0,0){};                 
           \node[T,label=above:$s_1$](s1) at (0,2){}; 
                    
         \node[txt]() at (1.4,1.6){$Y_2$};
               \node[txt]() at (.35,1.35){$P_1$};
               \node[txt]() at (1.6,.6){$Y_3$};
         \node[txt]() at (1,-.9){(i)};            
     	\end{tikzpicture}
	\hskip.5truecm
 \begin{tikzpicture}

   \draw[double,line width=.5pt]   (0,2)--(0,1)--(2,1) (1,2)--(1,0); 
   \draw[double, line width=.3pt,snake=zigzag]  (0,2) -- (2,2)--(2,0); 
   \draw[line width=2pt] (0,0) -- (0,1);
       \draw[dashed] (0,0)--(2,0);  
   \node[B]() at (2,2){};      \node[A]() at (1,0){};  
   \node[A]() at (2,0){};  \node[A]() at (2,1){}; 
   \node[B]() at (1,2){};    
         \node[B]() at (1,1){};           
              \node[B]() at (0,2){}; 
                    \node[B]() at (2,1){}; 
                    \node[B]() at (2,0){};               
                     \node[B] () at (1,0) {};  
            
             \node[T,label=left:$s_1$](s1) at (0,1){};                                 
         \node[T,label=left:$t_1$](t1) at (0,0){};                 
  
         \node[txt]() at (.3,.5){$P_1$};
               \node[txt]() at (1.4,.6){$Y_3$};
               \node[txt]() at (1.6,1.6){$Y_2$};
         \node[txt]() at (1,-.9){(ii)};            
     	\end{tikzpicture}
		\hskip.7truecm
	\begin{tikzpicture}

   \draw[double,line width=.5pt]  (0,0)--(2,0)-- (2,2)  (0,1) -- (2,1); 
   \draw[double, line width=.3pt,snake=zigzag]   (0,0)--(0,1) -- (0,2) -- (1,2) -- (2,2); 
   \draw[line width=2pt] (1,0) -- (1,2);
         
   \node[B]() at (2,2){};      
   \node[B]() at (0,2){};   \node[A]() at (0,0){};  \node[A]() at (2,2){}; 
         \node[B]() at (1,1){};      \node[A]() at (2,0){};   
              \node[B]() at (0,1){}; 
                    \node[B]() at (2,1){};                                             
                   \node[B]() at (2,0){};               
      \node[B] () at (0,0) {};                     
         \node[T,label=above:$s_1$](s1) at (1,2){};  
                       \node[T,label=below:$t_1$](t1) at (1,0){};       
         \node() at (.4,1.6){$Y_2$};
               \node() at (1.3,1.3){$P_1$};
                 \node() at (1.6,.4){$Y_3$};
         \node() at (1,-.9){(iii)};            
     	\end{tikzpicture}	

	\caption{$s_1\in S$ and $t_1\in A\cup B$}
	\label{pi1separated}
		\end{center}
	\end{figure}
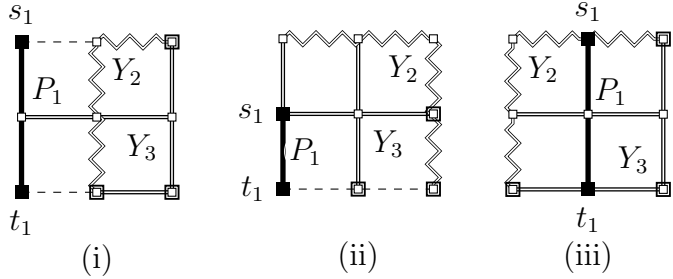
	
In each case, both difference sets,  $Y_2- Y_3$ and $Y_3- Y_2$, contain one 
vertex that can hold a terminal of $\pi_0$, furthermore, the anchor sets of the clamps are disjoint, thus Proposition \ref{partition} applies.

Case 3.2: $\pi_1\subset Q-A$ (or  $\pi_1\subset Q-B$).
For $s_1=(1,2)$ and $t_1=(2,3)$, let $P_1=(s_1,x_1,t_1)$. Define 
the clamp $Y_2$ to be the $4$-path $(1,3)-(1,2)-x_1-(3,2)$ anchored at its end vertices, and let clamp $Y_3$ be the complement of $P_1\cup Y_2$ anchored at $\{x_0\}$ (see Fig.\ref{pi1last} (i)). 

For $s_1\in \{(2,1),(2,2)\}$ and $t_1=(1,3)$ we define $P_1$ as the $s_1,t_1$-path contained in $(A(2)\cap Q)\cup\{t_1\}$. Let $Y_2$ and $Y_3$ be the clamps as before (see Fig.\ref{pi1last} (ii)). In both cases the requirements of Proposition \ref{partition} are satisfied.

For $s_1=(1,1)$ and $t_1=(2,3)$, we have two subcases. If $(1,3)\in \pi_0$, then
let $P_1=s_1 - (2,1) - x_1 - t_1$ be the linkage for $\pi_1$, let 
$Y_2=(1,3)-(1,2)-x_1-(3,2)$ considered as a clamp anchored at its end vertices, and define the clamp $Y_3$ to be the $6$-path
$A\cup B\cup \{(2,1)\}$ anchored at 
$\{x_0\}$ (see Fig.\ref{pi1last} (iii)). Since $(1,3)\in \pi_0\cap(Y_2\cap Y_3)$, there is a matching between the two clamps and $\pi_0$.
 	\begin{figure}[htp]
 \tikzstyle{T} = [rectangle, minimum width=.1pt, fill, inner sep=2.5pt]
\tikzstyle{B} = [rectangle, draw=black!, minimum width=1pt, fill=white, inner sep=1.5pt]
\tikzstyle{txt}  = [circle, minimum width=1pt, draw=white, inner sep=0pt]
\tikzstyle{Wedge} = [draw,line width=.7pt,-,black!100]
\tikzstyle{M} = [circle, draw=black!, minimum width=1pt, fill=white, inner sep=1pt]
\tikzstyle{A}  = [rectangle, line width=.7pt, draw=black, inner sep=2.5pt]

\begin{center}
\begin{tikzpicture}

   \draw[double,line width=.5pt]  (0,0)--(2,0)--(2,2)  (1,0) --(1,1) -- (0,1); 
   \draw[double, line width=.3pt,snake=zigzag]   (0,0)--(0,1) -- (0,2) -- (1,2) -- (2,2); 
   \draw[line width=2pt] (1,2) -- (1,1)-- (2,1);
            
         \node[B]() at (1,1){};       
              \node[B]() at (0,1){}; 
                    \node[B]() at (2,2){};                                             
                   \node[B]() at (2,0){};               
      \node[B] () at (0,0) {};    
        \node[B]() at (0,2){};   \node[A]() at (0,0){};  
          \node[B]() at (1,0){};  
            \node[A]() at (2,2){};   \node[A]() at (2,0){};  
            \node[T,label=above:$s_1$](s1) at (1,2){}; 
            \node[T,label=right:$t_1$](t1) at (2,1){};                                        
                                  
         \node[txt]() at (.4,1.6){$Y_2$};\node() at (.75,.75){$x_1$};
               \node[txt]() at (1.4,1.4){$P_1$};
                 \node[txt]() at (1.6,.4){$Y_3$};
         \node[txt]() at (1,-.9){(i)};            
     	\end{tikzpicture}	
	\begin{tikzpicture}

   \draw[double,line width=.5pt]  (0,0)--(2,0)--(2,1)  (1,2) --(1,0)  (0,1)--(1,1); 
   \draw[double, line width=.3pt,snake=zigzag]   (0,0)--(0,1) -- (0,2) -- (1,2) -- (2,2); 
   \draw[line width=2pt] (1,1) -- (2,1)-- (2,2);
            
   \node[B]() at (0,2){};  
         \node[T]() at (1,1){};      \node[A]() at (2,0){};   
          \node[A]() at (2,2){};   \node[A]() at (0,0){};  
              \node[B]() at (0,1){}; 
                    \node[B]() at (2,1){};                                             
                   \node[B]() at (2,0){};               
      \node[B] () at (0,0) {};    
        \node[B](s1) at (1,2){};
          \node[B]() at (1,0){};   
          
            \node[T,label=above:$t_1$](t1) at (2,2){};                                        
               \node[txt]() at (.75,.75){$s_1$};    
                          
         \node[txt]() at (.4,1.6){$Y_2$};
               \node[txt]() at (1.6,1.4){$P_1$};
                 \node[txt]() at (1.4,.4){$Y_3$};
         \node[txt]() at (1,-.9){(ii)};            
     	\end{tikzpicture}
	\hskip.5cm	
 	\begin{tikzpicture}

   \draw[double,line width=.5pt] (0,1)--(0,0)--(1,0) -- (2,0) -- (2,2) (0,1)--(2,1); 
   \draw[double, line width=.3pt,snake=zigzag]   (1,0)-- (1,2) -- (2,2); 
   \draw[line width=2pt] (0,2)--(0,1)--(2,1) ;    
     \draw[dashed] (0,2)--(1,2);
         
          \node[B]() at (0,1){};        \node[A]() at (2,0){};      
          \node[B]() at (1,0){}; 
           \node[B]() at (0,0){}; 
          \node[B] () at (1,1) {};                                
           \node[T]() at (2,2){};              
            \node[B]() at (2,0){};                          
         \node[B]() at (1,2){};    \node[A] () at (1,0){};
         
             \node[T,label=right:$t_1$](t1) at (2,1){};                 
           \node[T,label=above:$s_1$](s1) at (0,2){}; 
                    
         \node[txt]() at (1.4,1.6){$Y_2$};
               \node[txt]() at (.35,1.35){$P_1$};
               \node[txt]() at (1.6,.4){$Y_3$};
         \node[txt]() at (1,-.9){(iii)};            
     	\end{tikzpicture}
\begin{tikzpicture}

   \draw[double,line width=.5pt] (0,0)-- (2,0)  (1,2) --(1,0); 
   \draw[double, line width=.3pt,snake=zigzag]   (0,0)--(0,1)--  (2,1); 
   \draw[line width=2pt] (0,2)--(2,2)--(2,1) ;    
     \draw[dashed] (0,2)--(0,1) (2,0)--(2,1);
         
          \node[B]() at (0,1){};        \node[A]() at (0,0){};      
          \node[B]() at (1,0){}; 
           \node[B]() at (0,0){}; 
          \node[B] () at (1,1) {};                                
           \node[B]() at (2,2){};              
            \node[B]() at (2,0){};                          
         \node[B]() at (1,2){};    \node[A] () at (2,0){};
         
             \node[T,label=right:$t_1$](t1) at (2,1){};                 
           \node[T,label=above:$s_1$](s1) at (0,2){}; 
                    
         \node[txt]() at (1.6,1.6){$P_1$};
               \node[txt]() at (.4,.6){$Y_2$};
               \node[txt]() at (1.4,.4){$Y_3$};
         \node[txt]() at (1,-.9){(iv)};            
     	\end{tikzpicture}

	\caption{$\pi_1\subset Q- A$}
	\label{pi1last}
		\end{center}
	\end{figure}
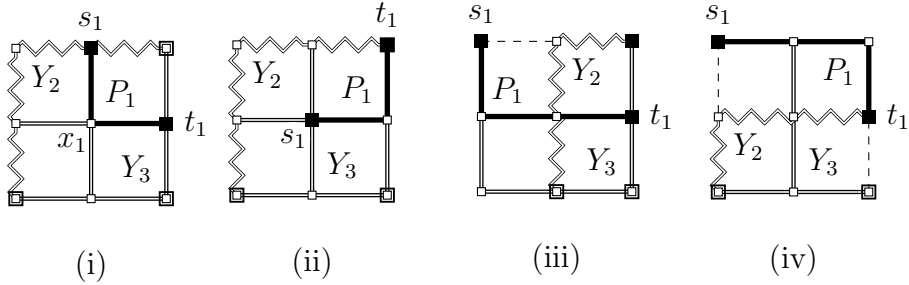
If $(1,3)\not\in \pi_0$, then we define
$P_1$ to be the $4$-path $s_1 - (1,2) - (1,3) - t_1$; let $Y_2=(3,1) - (2,1) - x_1 - (2,3) $  be a clamp anchord at its end vertices, and let $Y_3=(B(2)\cap Q)\cup A$ be a clamp anchored at $\{x_0\}$ (see Fig.\ref{pi1last} (iv)). In each of the difference sets 
$Y_2-Y_3$ and $Y_3-Y_2$ there is one vertex which can hold a terminal of $\pi_0$, thus Proposition \ref{partition} applies. 	
\end{proof}

\end{document}